\documentclass{amsproc}
\usepackage{euscript, graphicx, epstopdf}
\usepackage{cases}
\usepackage{mathrsfs}
\usepackage{bbm}
\usepackage{amssymb}
\usepackage{txfonts}
\usepackage{amscd}
\usepackage{amsfonts,latexsym,amsmath,amsxtra,mathdots,amssymb,latexsym,mathabx}
\usepackage[all,cmtip]{xy}
\usepackage{color}
\usepackage{colordvi}
\usepackage{multicol}
\usepackage{hyperref}
\usepackage{tikz}
\usepackage{float}
\usepackage{setspace, floatflt}

\allowdisplaybreaks

\newcommand{\Tr}{{{\rm Tr}}}

\newcommand{\BC}{{\mathbb {C}}}

 \newcommand{\BR}{{\mathbb {R}}}

 \newcommand{\BZ}{{\mathbb {Z}}}

 \newcommand{\GL}{{\mathrm{GL}}}
 \renewcommand{\Im}{{\mathrm{Im}\,}}

\newcommand{\PGL}{{\mathrm{PGL}}} 
\renewcommand{\Re}{{\mathrm{Re}\,}}

\def\-{^{-1}}

\def\-{^{-1}}

\def\lp {\left (}
\def\rp {\right )}

\def\Voronoi{Vorono\" \i \hskip 3 pt}
\def\SS{\mathscr S}
\def\boldJ {\boldsymbol J}

\def\nwedge {\hskip - 2 pt \wedge \hskip - 2 pt }

\makeatletter
\g@addto@macro\normalsize{\setlength\abovedisplayskip{3pt}}
\makeatother

\makeatletter
\g@addto@macro\normalsize{\setlength\belowdisplayskip{3pt}}
\makeatother

\newcommand{\delete}[1]{}

 \newcommand{\SL}{{\mathrm{SL}}}

\theoremstyle{plain}

\newtheorem{thm}{Theorem}[section] \newtheorem{cor}[thm]{Corollary}
\newtheorem{lem}[thm]{Lemma}  \newtheorem{prop}[thm]{Proposition}

\numberwithin{equation}{section}
\newtheorem{acknowledgement}{Acknowledgements}

\begin{document}

	\title[On the Fourier transform of Bessel functions over complex numbers---I]{On the Fourier Transform of Bessel Functions over \\ Complex Numbers---I: the Spherical Case}
	
	\author{Zhi Qi}
	
	\address{Department of Mathematics\\ Rutgers University\\Hill Center\\110 Frelinghuysen Road\\Piscataway, NJ 08854-8019\\USA}
	\email{zhi.qi@rutgers.edu}
	

	\begin{abstract}
		In this note, we  prove a formula for the Fourier transform of spherical Bessel functions over   complex numbers, viewed as the complex analogue of the classical formulae of Hardy and Weber. The formula has strong representation theoretic motivations in the Waldspurger correspondence over the complex field.
	\end{abstract}
	
	\maketitle

\section{Introduction}

\subsection{Representation Theoretic Motivations}

It is observed in the article \cite{BaruchMao-Real} of Baruch and Mao that two classical identities due to Weber and Hardy on  Bessel functions can be used to realize the Shimura-Waldspurger correspondence between representations of $\PGL_2 (\BR)$ and genuine representations of $\widetilde {\SL}_2 (\BR)$, 
when the involved  Bessel functions are attached to such representations in a certain way. 

First, we have the following formula of Weber (see \cite[1.13 (25), 2.13 (27)]{ET-I} and \cite[13.3 (5)]{Watson}),
\begin{equation}\label{0eq: Weber's formula}
\int_0^\infty \frac 1 {\sqrt x}  J_{\nu} \lp 4 \pi \sqrt x\rp e \lp {\pm x y} \rp   {d x}  = \frac 1 {\sqrt {2 y} } e \lp {\mp  \lp \frac 1 {2 y} - \frac 1 8  \nu - \frac 1 8 \rp} \rp J_{\frac 1 2 \nu} \lp \frac {\pi} y \rp,
\end{equation}
for $y > 0$,  where $e(x) = \exp \lp {2\pi i x} \rp$ and $J_\nu (x)$ is the  Bessel function of the first kind of order $\nu$. This formula is valid when  $\Re \nu > - 1$. In other words, the Fourier transform of $ x^{- \frac 1 2}  J_{\nu} \big( 4 \pi   x^{\frac 1 2} \big)  $ is equal to $ x^{-\frac 1 2} J_{\frac 1 2 \nu} \lp \pi  x^{-1} \rp  $ up to an exponential factor. 
Second, the identity of Hardy is of the same nature, which involves the modified Bessel function $K_{\nu}$ of the second kind of order $\nu$, (see \cite[1.13 (28), 2.13 (30)]{ET-I})
\begin{equation}\label{0eq: Hardy's formula}
\begin{split}
\int_0^\infty \frac 1 {\sqrt x}  K_{\nu} \lp 4 \pi \sqrt x\rp e \lp {\pm x y} \rp   {d x} & =   - \frac {\pi} {2 \sin ( \pi \nu)}  \frac 1 {\sqrt {2 y} } e \lp {\pm  \lp \frac 1 {2 y} + \frac 1 8 \rp} \rp \\
& \lp e \lp \pm \frac 1 8 \nu \rp J_{\frac 1 2 \nu} \lp \frac {\pi} y \rp - 
 e \lp \mp \frac 1 8 \nu \rp J_{- \frac 1 2 \nu} \lp \frac {\pi} y \rp \rp,
\end{split}
\end{equation}
with $|\Re \nu | < 1$.

Taking $\nu = 2 k - 1$  in \eqref{0eq: Weber's formula}, with $k$ a positive integer, 
the Bessel function of order $2 k - 1$, respectively $k - \frac 1 2$, is attached to a discrete series representation of $\PGL_2 (\BR)$, respectively $\widetilde {\SL}_2 (\BR)$. Thus,  in this case, \eqref{0eq: Weber's formula} should be interpreted as the local ingredient at the real place of the correspondence due to  Shimura, Shintani and Waldspurger between cusp forms of weight $2 k$ and cusp forms of weight $ k + \frac 1 2$. Likewise, two formulae \eqref{0eq: Weber's formula} and \eqref{0eq: Hardy's formula}, with $\nu = 2 i t$ for $t $ real or $\nu = 2 t$ for $0 < t < \frac 1 2$, may be combined to realize the Shimura-Waldspurger correspondence for principal or complementary series representations. See \cite{BaruchMao-Real}.

In the framework of the relative trace formula of Jacquet, the theory in \cite{BaruchMao-Real} along with the corresponding non-archimedean theory in \cite{BaruchMao-NA} is used in \cite{BaruchMao-Global} to produce a Waldspurger-type formula over a {\it totally real} number field.

\subsection{Main Theorem}

In this note, we shall consider the Fourier transform of spherical Bessel functions over $\BC$. Such Bessel functions are associated to spherical 
irreducible representations of $\PGL_2 (\BC)$. According to \cite[\S 15.3, 18.2]{Qi-Bessel},  we define
\begin{equation}\label{1eq: n=2, C}
\boldJ_{\mu} \lp z \rp = \frac {2 \pi^2} {\sin (2\pi \mu)}  \lp J_{- 2\mu   } \lp 4 \pi \sqrt z \rp J_{- 2\mu    } \big( 4 \pi \sqrt {\overline z} \big)    - J_{ 2\mu   } \lp 4 \pi \sqrt z \rp J_{ 2\mu }  \big( 4 \pi \sqrt {\overline z} \big)     \rp. 
\end{equation} 
It is understood that the right hand side should be in its limit form when the order $ \mu $ is a half integer.
Moreover, the expression in \eqref{1eq: n=2, C} is independent of the choice of $\arg z$.

We shall prove the following analogue of the  formulae of Weber and Hardy (\ref{0eq: Weber's formula}, \ref{0eq: Hardy's formula}).

\begin{thm}\label{thm: main 1} 
	Suppose that $|\Re \mu| < \frac 1 2$. 
	We have the identity
	\begin{equation}\label{eq: main 1}
	\begin{split}
	 \int_{0}^{2 \pi} \int_0^\infty \boldJ_{ \mu } \lp x e^{i\phi} \rp  e (- 2 x y \cos (\phi + \theta) )    d x  d \phi  
	= \frac {1} {4 y}   e\lp \frac {\cos \theta } y \rp  \boldJ_{  \frac 12 \mu } \lp  \frac 1 {
		16 y^2  e^{2 i \theta}  } \rp ,
	\end{split}
	\end{equation}
	with $y \in (0, \infty)$ and $\theta \in [0, 2 \pi)$. 
\end{thm}

\subsection{Remarks}
Spherical Bessel functions for $\SL_2 (\BC)$ were first discovered by Miatello and Wallach \cite{M-W-Kuz} in the study of the spherical Kuznetsov trace formula for real semisimple groups of real rank one.   Nonspherical Bessel functions for $\SL_2 (\BC)$ were found by Bruggeman, Motohashi and Lokvenec-Guleska \cite{B-Mo,B-Mo2} in generalizing the Kuznetsov trace formula to the nonspherical setting. In an entirely different way, they were recently rediscovered by the author \cite{Qi-Bessel} as the $\GL_2 (\BC)$ example of the  Bessel functions arising in the \Voronoi summation formula for $\GL_n (\BC)$.    In general, for $\mu \in \BC $ and $m \in \BZ$, the Bessel function of index $(\mu, m)$ is associated with the principal series representation  of $\SL_2 (\BC)$  induced from the character $\chiup_{\,\mu,\, m}\begin{pmatrix}
a & \\
& a\- 
\end{pmatrix} = |a|^{4 \mu} (a/|a|)^{  m}$.  When $ m = 0$, we are in the spherical case.

From the viewpoint of representation theory, Bessel functions for $\GL_2 (\BC)$ are defined in  parallel with those for $\GL_2 (\BR)$. The  Bessel function attached to an irreducible unitary representation  of $\GL_2 (\BR)$ or $\GL_2 (\BC)$ is defined as the integral kernel of a kernel formula for the Weyl element action on the associated Kirillov model. Such a kernel formula  for $\GL_2 (\BR)$ and $\GL_2 (\BC)$ lies in the center of the representation theoretic approach to the Kuznetsov trace formula; see \cite{CPS} and \cite{Qi-Kuz}. In the case of $\GL_2 (\BR)$ or $\SL_2 (\BR)$, there are three proofs of the kernel formula in \cite[\S 8]{CPS},  \cite{Mo-Kernel} and \cite[Appendix 2]{BaruchMao-Real}, while the  formula for $\GL_2 (\BC)$ or $\SL_2 (\BC)$ and its proofs may be found in \cite{B-Mo-Kernel2,Mo-Kernel2}, \cite{Baruch-Kernel} and \cite[\S 17, 18]{Qi-Bessel}. 


The identity \eqref{eq: main 1} 
may be interpreted from the viewpoint of the Shimura-Waldspurger correspondence between spherical unitary representations of $\PGL_2 (\BC)$ and   ${\SL}_2 (\BC)$. This will enable  us to extend the Waldspurger formula to an {\it arbitrary} number field for the spherical case. 

Finally, we remark that a similar identity should be expected to hold for nonspherical Bessel functions for $\PGL_2 (\BC)$.\footnote{At the time when the first draft for this note was written in February 2015, the author was not able to prove such an identity. The approach presented here fails to work in  the nonspherical case because there is no formula for classical Bessel functions that would be of any help. Some other attempts also resulted in failure.
	Recently, the author proved the general formula in \cite{Qi-II-G}, using an indirect method combining stationary phase and differential equations, along with a radial exponential integral formula. It turns out that the inductive arguments  in \cite{Qi-II-G} start  from the case $ m = \pm 2$ but not $m = 0$ ($m$    even), so the nonspherical case should {\it not} be considered as a straightforward extension of the spherical case.}







\section{Preliminaries on Classical Bessel Functions}

Let $J_{\nu} (z)$, $Y                                                                                                                                                                                                                                                                                                                                                                                                                                                                                                                                                                                                                                                                                                                                                                                                                                                                                                                                                                                                                                                                                                                                                                                                                                                                                                                                                                                                                                                                                                                                                                                                                                                                                                                                                                                                                                                                                                                                                                                                                                                                                                                                                                                                                                                                                                                                                                                                                                                                                                                                                                                                                                                                                                                                                                                                                                                                                                                                                                                                                                                                                                                                                                                                                                                                                                                                                                                                                                                                                                                                                                                                                                                                                                                                                                                                                                                                                                                                                                                                                                                                                                                                                                                                                                                                                                                                                                                                                                                                                                                                                                                                                                                                                                                                                                                                                                                                                                                                                                                                                                                                                                                                                                                                                                                                                                                                                                                                                                                                                                                                                                                                                                                                                                                                                                                                                                                                                                                                                                                                                                                                                                                                                                                                                                                                                                                                                                                                                                                                                                                                                                                                                                                                                                                                                                                                                                                                                                                                                                                                                                                                                                                                                                                                                                                                                                                                                                                                                                                                                                                                                                                                                                                                                                                                                                                                                                                                                                                                                                                                                                                                                                                                                                                                                                                                                                                                                                                                                                                                                                                                                                                                                                                                                                                                                                                                                                                                                                                                                                                                                                                                                                                                                                                                                                                                                                                                                                                                                                                                                                                                                                                                                                                                                                                                                                                                                                                                                                                                                                                                                                                                                                                                                                                                                                                                                                                                                                                                                                                                                                                                                                                                                                                                                                                                                                                                                                                                                                                                                                                                                                                                                                                                                                                                                                                                                                                                                                                                                                                                                                                                                                                                                                                                                                                                                                                                                                                                                                                                                                                                                                                                                                                                                                                                                                                                                                                                                                                                                                                                                                                                                                                                                                                                                                                                                                                                                                                                                                                                                                                                                                                                                                                                                                                                                                                                                                                                                                                                                                                                                                                                                                                                                                                                                                                                                                                                                                                                                                                                                                                                                                                                                                                                                                                                                                                                                                                                                                                                                                                                                                                                                                                                                                                                                                                                                                                                                                                                                                                                                                                                                                                                                                                                                                                                                                                                                                                                                                                                                                                                                                                                                                                                                                                                                                                                                                                                                                                                                                                                                                                                                                                                                                                                                                                                                                                                                                                                                                                                                                                                                                                                                                                                                                                                                                                                                                                                                                                                                                                                                                                                                                                                                                                                                                                                                                                                                                                                                                                                                                                                                                                                                                                                                                                                                                                                                                                                                                                                                                                                                                                                                                                                                                                                                                                                                                                                                                                                                                                                                                                                                                                                                                                                                                                                                                                                                                                                                                                                                                                                                                                                                                                                                                                                                                                                                                                                                                                                                                                                                                                                                                                                                                                                                                                                                                                                                                                                                                                                                                                                                                                                                                                                                                                                                                                                                                                                                                                                                                                                                                                                                                                                                                                                                                                                                                                                                                                                                                                                                                                                                                                                                                                                                                                                                                                                                                                                         _{\nu} (z)$, $H^{(1,2)}_{\nu}  (z) $ denote the three kinds of Bessel functions of order $\nu$ and $I_{\nu}  (z) $ the modified Bessel function of the first kind of order $\nu$. 

$J_{\nu} (z)$ is defined by the series    (see \cite[3.1 (8)]{Watson})
\begin{equation}\label{2def: series expansion of J}
	J_{\nu} (z) = \sum_{n=0}^\infty \frac {(-)^n \lp \frac 1 2 z \rp^{\nu+2n } } {n! \Gamma (\nu + n + 1) }.
\end{equation}
We have the following connection formulae  (see \cite[3.61 (1, 2, 3, 4),  3.1 (8), 3.7 (2)]{Watson})
\begin{align}
	\label{2eq: J and H} 
	& J_\nu (z) = \frac {H_\nu^{(1)} (z) + H_\nu^{(2)} (z)} 2, \hskip 30 pt 
	J_{-\nu} (z) =  \frac {e^{\pi i \nu} H_\nu^{(1)} (z) + e^{-\pi i \nu} H_\nu^{(2)} (z) } 2, \\
	\label{2eq: Y and J} & Y_{\nu} (z) = \frac {J_{\nu} (z) \cos (\pi \nu)  - J_{- \nu} (z)} {\sin (\pi \nu)}, \hskip  10 pt Y_{- \nu} (z) = \frac {J_{\nu} (z)   - J_{- \nu} (z) \cos (\pi \nu)} {\sin (\pi \nu)}, \\
	\label{2eq: I and J} & I_{\nu} \big( e^{\pm \frac 1 2 \pi i } z \big)  = e^{\pm \frac 1 2 \pi i \nu } J_{\nu} (z)  .
\end{align}
We have the asymptotics of $H^{(1)}_{\nu} (z)$ and $H^{(2)}_{\nu} (z)$ (see \cite[7.2 (1, 2)]{Watson}),
\begin{equation}\label{2eq: asymptotic H (1)}
	H^{(1)}_{\nu} (z) = \lp \frac 2 {\pi z} \rp^{\frac 1 2} e^{ i \lp z - \frac 1 2 {\pi \nu}   - \frac 1 4 \pi    \rp }  \lp 1  + \frac {   1   - 4 \nu^2} {8 i z}  + O \lp \frac 1 { |z|^{  2} } \rp \rp,
\end{equation}
\begin{equation}\label{2eq: asymptotic H (2)}
	H^{(2)}_{\nu} (z) = \lp \frac 2 {\pi z} \rp^{\frac 1 2} e^{ - i \lp z - \frac 1 2{\pi \nu}    - \frac 1 4 \pi   \rp } \lp 1 - \frac {   1   - 4 \nu^2} {8 i z}  + O \lp \frac 1 { |z|^{  2} } \rp \rp,
\end{equation}
of which \eqref{2eq: asymptotic H (1)} is valid when $z$ is such that $- \pi +   \delta \leqslant \arg z \leqslant 2 \pi -  \delta$,  and  \eqref{2eq: asymptotic H (2)} when $- 2 \pi +  \delta \leqslant \arg z \leqslant   \pi -  \delta$,   $  \delta  $ being any positive acute angle. 

\section{Two Formulae for Classical Bessel Functions}

In this section, we shall give two formulae for  classical Bessel functions that will be used for the proof of Theorem \ref{thm: main 1}.

\subsection{ }\label{sec: 1st formula}

We have Weber's second exponential integral formula \cite[13.31 (1)]{Watson}
\begin{equation}\label{eq: Weber integral, 0}
	\int_0^\infty J_{\nu} (a x) J_{\nu} (b x) \exp \big(- p^2 x^2\big)  x d x = \frac 1 {2 p^2} \exp \lp - \frac {a^2 + b^2} {4 p^2} \rp I_{\nu} \lp \frac {a b} {2 p^2} \rp.
\end{equation}
It is required that $\Re \nu > - 1$ and $|\arg p| < \frac 1 4\pi$ to secure absolute convergence, but $a$ and $b$ are unrestricted nonzero complex numbers. It follows that
\begin{equation}\label{eq: Weber integral}
	\begin{split}
		\int_0^\infty \big( J_{-\nu} \lp a \sqrt x \rp J_{-\nu} \lp \overline a \sqrt x \rp & -  J_{ \nu} \lp a \sqrt x \rp J_{ \nu} \lp \overline a \sqrt x \rp \big)  \exp \left(- p^2 x \right)  d x \\
		=  & \ \frac 1 { p^2}  \exp   \bigg( - \frac {a^2 + \overline a^2} {4 p^2} \bigg)    \lp  I_{-\nu} \lp \frac {|a|^2} {2 p^2} \rp - I_{ \nu} \lp \frac {|a|^2} {2 p^2} \rp \rp,
	\end{split}
\end{equation}
if  $|\Re \nu | < 1$,    $|\arg p| < \frac 1 4\pi$ and  $a \neq 0 $.
We claim that \eqref{eq: Weber integral}  is still valid even if $\arg p = \pm \frac 1 4 \pi$ in the sense that the left hand side of \eqref{eq: Weber integral} remains convergent and converges to the value on the right.

Since the expressions on both sides of \eqref{eq: Weber integral} are indeed independent on the argument of $a$ modulo $\pi$, we shall assume that $- \frac 1 2 \pi <  \arg a  \leqslant \frac 1 2 \pi $.

To prove the convergence, we first partition the integral in \eqref{eq: Weber integral} into two integrals over the intervals $(0, 1]$ and $[1, \infty)$ respectively. Since $|\Re \nu| < 1$, the first integral  is absolutely convergent due to the series expansions of $J_{\nu} (z)$ and $J_{- \nu} (z)$ at zero (see \eqref{2def: series expansion of J}).
As for the convergence of the second integral, using the connection formulae \eqref{2eq: J and H},
we write
\begin{align*}
	 J_{-\nu} \lp a \sqrt x \rp & J_{-\nu} \lp \overline a \sqrt x \rp - J_{ \nu} \lp a \sqrt x \rp J_{ \nu} \lp \overline a \sqrt x \rp \\ 
	= \ & 
	\frac {i \sin (\pi \nu) } 2 \lp e^{\pi i \nu} H^{(1)}_{ \nu} \lp a \sqrt x \rp H^{(1)}_{ \nu} \lp \overline a \sqrt x \rp -  e^{- \pi i \nu} H^{(2)}_{ \nu} \lp a \sqrt x \rp H^{(2)}_{ \nu} \lp \overline a \sqrt x \rp \rp.
\end{align*}
Therefore, in view of the   asymptotics  of $H^{( 1 )}_{ \nu} \lp z \rp $ and  $H^{( 2 )}_{ \nu} \lp z \rp $ at infinity (\ref{2eq: asymptotic H (1)}, \ref{2eq: asymptotic H (2)}),
with $ | \arg z | \leqslant \frac 1 2 \pi  $, we are reduced to the convergence of the following integrals
\begin{equation*}
	\int_1^\infty \frac 1 {\sqrt x} \exp \lp {\pm i (a + \overline a) \sqrt x}   {- p^2 x} \rp d x, \hskip 10 pt \int_1^\infty \frac 1 {  x} \exp \lp {\pm i (a + \overline a) \sqrt x}   {- p^2 x} \rp d x.
\end{equation*}
By partial integration, the former  turns into
\begin{align*}
	& \lp \frac 1 {p^2} \pm \frac {i (a + \overline a)} {2 p^4} \rp \exp   \lp \pm i (a + \overline a)  {- p^2  } \rp \\
	& - \lp \frac 1 {2p^2} + \frac {(a + \overline a )^2} {4 p^4} \rp \int_1^\infty   x^{-\frac 3 2}  \exp \lp {\pm i (a + \overline a) \sqrt x}   {- p^2 x} \rp d x \\
	& \mp \frac {i (a + \overline a )} {2 p^4} \int_1^\infty   x^{- 2}  \exp \lp {\pm i (a + \overline a) \sqrt x}   {- p^2 x} \rp d x,
\end{align*}
in which the integrals converge absolutely whenever $|\arg p| \leqslant \frac 1 4\pi$. The convergence of the latter may be proven in the same way. This completes the proof of the convergence. 

With the above arguments, it is easy to verify that the left hand side of \eqref{eq: Weber integral} gives rise to a continuous function  on the sector $\left\{ p : |\arg p| \leqslant \frac 1 4\pi \right\}$.  Then follows the second assertion in the claim. 

\vskip 7 pt

Rescaling $a$ by the factor $4 \pi$, letting $p^2 = 2 \pi e^{\mp \frac 1 2 \pi i} c$ and using the identity \eqref{2eq: I and J} to turn the $I_{\pm \nu}$   into $J_{\pm \nu}$, the claim above yields the following lemma.

\begin{lem}\label{2lem: the 1st Bessel identity}
	Suppose that $|\Re \nu| < 1$,  $a \neq 0$ and $c > 0$. Then
	\begin{align*}
		 \int_0^\infty & \lp J_{-\nu} \lp 4 \pi a \sqrt x \rp J_{-\nu} \lp 4 \pi \overline a \sqrt x \rp - J_{ \nu} \lp  4 \pi a \sqrt x \rp J_{ \nu} \lp  4 \pi \overline a \sqrt x \rp \rp e \left( \pm c x \right)  d x \\
		& \hskip 20 pt =   \mp \frac 1 {2 \pi i c} e \bigg(  \mp \frac {a^2 + \overline a^2} {c} \bigg) \lp e^{\mp \frac 1 2 \pi i \nu} J_{-\nu} \lp  \frac  {4 \pi |a|^2} {c} \rp - e^{\pm \frac 1 2 \pi i \nu} J_{ \nu} \lp \frac {4 \pi |a|^2} {c} \rp \rp.
	\end{align*}
\end{lem}


\subsection{ }\label{sec: 2nd formula}

According to \cite[8.7 (22)]{ET-II}, we have the following formula
\begin{equation}\label{2eq: 2nd identity, EMOT}
	\begin{split}
		\int_0^1 J_{\nu} (a x) \frac {\sin \lp b  \sqrt {1-x^2} \rp} {\sqrt {1 - x^2 }} d x & - \int_1^\infty J_{\nu} (a x) \frac {\exp \lp - b \sqrt {x^2 - 1} \rp} {\sqrt {x^2 - 1}} d x \\
		& =   \frac \pi 2 \, J_{\frac 1 2 \nu} \lp\frac  { \sqrt {a^2 + b^2} - b } 2 \rp Y_{\frac 1 2 \nu} \lp\frac  { \sqrt {a^2 + b^2} + b  } 2 \rp.
	\end{split}
\end{equation}
It is assumed in  \cite{ET-II} that $\Re \nu > - 1$ and $a, b > 0 $. We shall prove however that  \eqref{2eq: 2nd identity, EMOT} is valid even if $\Re b \geqslant 0$, provided that $\Re \nu > - 1$ and $a > 0$. 

Let  $\Re \nu > - 1$ and $a > 0$ be fixed. First of all, the first integral on the left hand side of \eqref{2eq: 2nd identity, EMOT} is absolutely convergent for any   complex number $b$, and, if one makes the additional assumption $\Re b \geqslant 0$, then so is the second. The formula \eqref{2eq: 2nd identity, EMOT} is a relation connecting functions of $b$ which are analytic on the right half plane $\left\{ b : \Re b  > 0 \right\}$ and extend continuously onto the  imaginary axis $\left\{ b : \Re b = 0 \right\}$, where the principal branch of the square root is assumed on the right hand side. Therefore, by the principle of analytic continuation, \eqref{2eq: 2nd identity, EMOT} remains valid whenever $\Re b  \geqslant 0 $. 

\vskip 7 pt 

We shall be particularly interested in the case when $\Re b = 0$ and $|\Im b | \leqslant a$. With the observation that the first integral in  \eqref{2eq: 2nd identity, EMOT} is odd with respect to $b$, along with the identities in \eqref{2eq: Y and J} that connect $Y_{\pm \frac 1 2 \nu}$ with $J_{ \frac 1 2 \nu}$ and $J_{ - \frac 1 2 \nu}$, we obtain the following lemma.
\begin{lem}
	\label{2lem: 2nd identity}
	Suppose that $|\Re \nu| < 1$, $a  > 0$ and $ - a \leqslant c \leqslant a$. If we  put $ w = \frac 1 2 \lp \sqrt {a^2  - c^2} + i c \rp$, then
	\begin{equation*}
		\begin{split}
			\int_1^\infty \lp  J_{- \nu} (a x) + J_{ \nu} (a x) \rp &  \frac {\cos \lp c \sqrt {x^2 - 1} \rp} {\sqrt {x^2 - 1}} d x \\
			= \ & \frac {\pi \cot \lp \frac 1 2\pi \nu \rp} 2  \lp J_{- \frac 1 2 \nu} \lp w \rp J_{- \frac 1 2 \nu} \lp \overline w \rp - J_{  \frac 1 2 \nu} \lp  w \rp J_{   \frac 1 2 \nu} \lp \overline w \rp  \rp.
		\end{split}
	\end{equation*}
	
\end{lem}

\section{Proof of Theorem \ref{thm: main 1}}

With the preparations in the last section, we are now ready to prove Theorem \ref{thm: main 1}. Indeed, we shall prove the following reformulation of Theorem \ref{thm: main 1}. 

\begin{prop}\label{3prop: unramified} 
Suppose that $|\Re \mu| < \frac 1 2$. 
The iterated double integral 
\begin{equation}\label{3eq: integral, Fourier}
\begin{split}
  \int_{0}^{2 \pi} \int_0^\infty \Big(   J_{- 2 \mu } \big(4 \pi e^{ \frac 1 2 i\phi} & \sqrt x \big)   J_{- 2 \mu  } \big( 4 \pi  e^{- \frac 1 2 i\phi} \sqrt x  \big) - \\
  & J_{2 \mu } \big(4 \pi e^{ \frac 1 2 i\phi} \sqrt x \big)   J_{2 \mu  } \big( 4 \pi  e^{- \frac 1 2 i\phi} \sqrt x  \big) \Big)
   e (\pm 2 x y \cos (\phi + \theta )  )  d x \, d \phi,
\end{split} 
\end{equation}
with $y \in (0, \infty)$ and $  \theta \in \left ( - \frac 1 2 \pi, \frac 1 2 \pi \right]$, is equal to
\begin{align*}
\frac { \cos \lp \pi \mu \rp} {2 y} e \lp \mp \frac {\cos \theta} y \rp
\lp J_{- \mu} \lp \frac {\pi} {y e^{  i \theta}} \rp J_{- \mu} \lp \frac {\pi} {y e^{ - i \theta}} \rp -
J_{ \mu} \lp \frac {\pi} {y e^{  i \theta}} \rp J_{ \mu} \lp \frac {\pi} {y e^{ - i \theta}} \rp  \rp.
\end{align*}
\end{prop}


\begin{proof}
	
	Our proof   combines the applications of Lemma \ref{2lem: the 1st Bessel identity}  to the radial integral over $  x$ and then Lemma \ref{2lem: 2nd identity} to the angular integral over $  \phi$.
	
We start with dividing the domain of   integration into four quadrants. Indeed, if we substitute $\phi$ by $\phi - \theta$, partition the  domain of  $\phi$ into four open intervals $\lp 0, \frac 1 2 \pi \rp$, $\lp \frac 1 2 \pi,   \pi \rp$,  $\lp \pi, \frac 3 2 \pi \rp$ and  $\lp \frac 3 2 \pi, 2 \pi \rp$, and make a suitable change of variables for each resulting integral, then the   integral \eqref{3eq: integral, Fourier} turns into the sum of four similar integrals, the first of which is
\begin{equation}\label{eq: double integral 1}
\begin{split}
\int_{0}^{\frac 1 2 \pi} \int_0^\infty \Big( J_{- 2 \mu }  \big(4 \pi & e^{ \frac 1 2 i(\phi - \theta)}  \sqrt x \big)   J_{- 2 \mu  } \big( 4 \pi  e^{- \frac 1 2 i (\phi - \theta)} \sqrt  x  \big) - \\
& J_{ 2 \mu } \big(4 \pi e^{ \frac 1 2 i(\phi - \theta)}  \sqrt x \big)   J_{ 2 \mu } \big( 4 \pi  e^{- \frac 1 2 i (\phi - \theta)} \sqrt  x  \big) \Big) e(\pm 2 x  y \cos \phi  )  d x \, d \phi .
\end{split}
\end{equation}
For $\phi \in \lp 0 , \frac 1 2 \pi \rp$, 
upon choosing $\nu = 2 \mu$, $a =  e^{ \frac 1 2 i(\phi - \theta)}$, $c = 2 y \cos \phi $ in Lemma \ref{2lem: the 1st Bessel identity}, we deduce that the  integral \eqref{eq: double integral 1} is equal to
\begin{align*}
& \mp \frac 1 {4 \pi i y} \int_{0}^{\frac 1  2 \pi} \frac 1 {\cos \phi}  e \lp  \mp \frac { \cos \lp \phi - \theta \rp } {   y \cos \phi} \rp \lp e^{\mp \pi i \mu} J_{ - 2 \mu } \lp  \frac {2 \pi  } {   y \cos \phi} \rp - e^{ \pm \pi i \mu} J_{ 2 \mu } \lp  \frac {2 \pi  } {   y \cos \phi} \rp \rp d \phi \\
& =  \mp \frac 1 {4 \pi i y} e \lp \mp \frac {\cos \theta} y \rp \int_{0}^{\frac 1  2 \pi} \frac 1 {\cos \phi}  e \lp \mp  \frac { \sin \theta \tan \phi } {   y } \rp \\
& \hskip 150 pt \lp e^{\mp \pi i \mu} J_{ - 2 \mu} \lp  \frac {2 \pi  } {   y \cos \phi} \rp - e^{ \pm \pi i \mu} J_{ 2 \mu} \lp  \frac {2 \pi  } {   y \cos \phi} \rp \rp d \phi.
\end{align*}
The other three integrals that are similar to \eqref{eq: double integral 1} can be computed in the same way. Summing these up, with certain exponential factors combined  into sine and cosine, the integral \eqref{3eq: integral, Fourier} turns into 
\begin{equation}\label{3eq: outer integral, 1}
\begin{split}
\frac {\sin (\pi \mu) } { \pi   y} e \lp \mp \frac {\cos \theta} y \rp \int_{0}^{\frac 1  2 \pi} \frac 1 {\cos \phi}  \cos  & \lp   \frac { 2 \pi \sin \theta \tan \phi } {   y } \rp  \\
& \lp J_{ - 2 \mu} \lp  \frac {2 \pi  } {   y \cos \phi} \rp + J_{ 2 \mu} \lp  \frac {2 \pi  } {   y \cos \phi} \rp \rp d \phi.
\end{split}
\end{equation}
Making the change of variables $t = 1/ \cos \phi$,  the   integral in \eqref{3eq: outer integral, 1} becomes
\begin{equation} \label{eq: intgral 3}
\begin{split}
\int_1^\infty \lp J_{- 2 \mu} \lp \frac {2 \pi} y t \rp + J_{ 2 \mu} \lp \frac {2 \pi} y t \rp \rp \cos \lp \frac {2 \pi \sin \theta} y \sqrt {t^2 - 1} \rp \frac {d t} {\sqrt {t^2 - 1}} .
\end{split}
\end{equation}
Applying Lemma \ref{2lem: 2nd identity} with $\nu = 2 \mu$, $a =   {2 \pi} / y$, $c = {2 \pi \sin \theta} / y$ and $w = \pi e^{i \theta} /y $, the integral \eqref{eq: intgral 3} is then equal to
\begin{align*}
\frac {\pi \cot \lp \pi \mu \rp} 2
\lp J_{- \mu} \lp \frac {\pi} {y e^{  i \theta}} \rp J_{- \mu} \lp \frac {\pi} {y e^{ - i \theta}} \rp -
 J_{\mu} \lp \frac {\pi} {y e^{  i \theta}} \rp J_{\mu} \lp \frac {\pi} {y e^{ - i \theta}} \rp  \rp.
\end{align*}
Taking  into account the factors in front of the integral in \eqref{3eq: outer integral, 1}, the proof is now complete.
\end{proof}

\begin{acknowledgement}
	The author would like to thank  the  referees for  their  helpful suggestions.
\end{acknowledgement}


	\bibliographystyle{alphanum}
	\bibliography{references}
	

\end{document}